\documentclass[10pt]{amsart}
\usepackage[T1]{fontenc}
\usepackage[utf8x]{inputenc}
\usepackage{amsmath}
\usepackage{amsmath, amssymb, amsfonts, amsthm, verbatim, dsfont, graphicx}
\usepackage{color}
\usepackage{mathrsfs}
\usepackage{amstext}
\usepackage{amsxtra}
\usepackage{amssymb}
\usepackage{amsgen}
\usepackage{amsfonts}
\usepackage{amsthm}
\usepackage{amscd}
\usepackage{latexsym}
\usepackage[margin=1in]{geometry}
\usepackage{cite}
\usepackage{setspace}
\usepackage{bbm}
\usepackage{cancel}
\usepackage{xcolor}
\usepackage{mathtools}
\usepackage{array}
\usepackage{tabu}

\newtheorem{theorem}{Theorem}[section]
\newtheorem{lemma}{Lemma}[section]
\newtheorem{proposition}{Proposition}[section]

\newtheorem{definition}{Definition}[section]

\theoremstyle{definition}

\theoremstyle{remark}
\newtheorem{remark}{Remark}[section]

\numberwithin{equation}{section}

\begin{document}

\title[Global Flows for Inviscid SQG]{
Global Flows with Invariant Measures\\
for the inviscid modified SQG equations}

\author[Nahmod]{Andrea R. Nahmod$^1$}
\address{$^1$ 
Department of Mathematics \\ University of Massachusetts\\ 710 N. Pleasant Street, Amherst MA 01003}
\email{nahmod@math.umass.edu}
\thanks{$^1$ The first author is funded in part by NSF DMS-1201443 and DMS-1463714.}

\author[Pavlovi\'c]{Nata\v{s}a Pavlovi\'c$^2$}
\address{$^2$ 
Department of Mathematics\\ 
University of Texas at Austin\\ 
2515 Speedway, Stop C1200\\
Austin, TX 78712}
\email{natasa@math.utexas.edu}
\thanks{$^2$ The second author is funded in part by NSF DMS-1516228.}

\author[Staffilani]{Gigliola Staffilani$^3$}
\address{$^3$ Department of Mathematics\\
Massachusetts Institute of Technology\\ 
77 Massachusetts Avenue, Cambridge, MA 02139}
\email{gigliola@math.mit.edu}
\thanks{$^3$ The third author is funded in part by NSF
DMS-1362509 and DMS-1462401.}

\author[Totz]{Nathan Totz$^4$}
\address{$^4$ 
Department of Mathematics \\ University of Massachusetts\\ 710 N. Pleasant Street, Amherst MA 01003}
\email{totz@math.umass.edu}
\thanks{$^4$ The fourth author is funded in part by NSF DMS-1612931.}

\date{May 3,2017}

\begin{abstract}
We consider the family known as modified or generalized surface quasi-geostrophic equations (mSQG) consisting of the classical inviscid surface quasi-geostrophic (SQG) equation together with a family of regularized active scalars given by introducing a smoothing operator of nonzero but possibly arbitrarily small degree. This family naturally interpolates between the 2D Euler equation and the SQG equation. For this family of equations we construct an invariant measure on a rough $L^2$-based Sobolev space and establish the existence of solutions of arbitrarily large lifespan for initial data in a set of full measure in the rough Sobolev space. 
\end{abstract}
\maketitle

\section{Introduction}

The inviscid surface quasi geostrophic (SQG) equation, 
\begin{equation}\label{sqg}
\begin{cases}
\theta_t + (u \cdot \nabla)\theta = 0, \; \; \; \; x \in \mathbb{M} \mbox{ and } t >0\\
u = \mathcal{R}^\perp \theta, 
\end{cases}
\end{equation}
with $\mathbb{M}$ being either $\mathbb{T}^2$ or $\mathbb{R}^2$ is by now a well known geophysical model in atmospheric sciences which has been systematically studied by Constantin, Majda, and Tabak who, in particular, developed an analogy with the 3D Euler equations; by Pierrehumbert, Held, and Swanson; by Held, Pierrehumbert, Garner and Swanson and others (see \cite{comata1, comata2, pierre, held} and references therein).
This equation has since attracted a lot of attention. Many interesting results describing the behavior of \eqref{sqg} have been obtained, see e.g. \cite{re, co, cofe, wu, caco}. In particular, Resnick \cite{re}, constructed global in time weak solutions to \eqref{sqg} for initial data $\theta_0 \in L^2(\mathbb{R}^2)$.  The question of global in time existence and uniqueness of 
strong solutions to \eqref{sqg} - or global regularity - however, is still an outstanding open problem, just as for the 3D Navier-Stokes and 3D Euler equations.  
The SQG equation has strong similarities to the vorticity formulation of the incompressible 2D Euler equation,  in which $\theta$ is replaced by the fluid vorticity $\omega$ and the relation between the active scalar $u$ and the solution $\omega$ is altered as follows:
\begin{equation}\label{euler}
\begin{cases}
\omega_t + (u \cdot \nabla)\omega = 0, \; \; \; \; x \in \mathbb{M} \mbox{ and } t >0\\
u = \mathcal{R}^\perp |D|^{-1} \omega, 
\end{cases}
\end{equation}
where again $\mathbb{M}$ could be either $\mathbb{T}^2$ or $\mathbb{R}^2$. 
Global existence of classical smooth solutions to \eqref{euler} has been known for a while; see e.g. \cite{yud} and references therein. 
On the other hand, Albeverio and Cruzeiro \cite{alcr} used probabilistic tools to prove the existence of global flows valued in Sobolev spaces $H^s(\mathbb T^2), s< -2$ via the invariance of (Gibbs) measure associated to the conservation of enstrophy,  $\|\omega\|_{L^2(\mathbb{T}^2)}$,  for the incompressible 2D Euler equation \eqref{euler}. 

In light of \cite{alcr} it is natural to ask whether a similar program can be carried out for the SQG equation \eqref{sqg}, which also has conservation of  `enstrosphy'  $\|\theta\|_{L^2(\mathbb{T}^2)}$ . Such global flows, if they exist, would be less regular than those constructed by Resnick. A cornerstone of the argument laid in \cite{alcr} is showing that the $H^s$-norm of nonlinearity of the equation is finite almost surely with respect to the associated invariant measure $\rho$ constructed based on the conservation of SQG enstrophy (say in $L^2_{\rho}$). 

This estimate is a key ingredient in the compactness argument used to construct the random flows. At this point it is 
worth noting that the SQG equation \eqref{sqg} is one derivative less regular than 2D Euler equation \eqref{euler}. And this loss of regularity turns out to be insurmountable for it actually renders each Fourier mode of the nonlinearity of the SQG equation \eqref{sqg} infinite almost surely with respect to $\rho$, (c.f. Section \ref{subsec-expectation}).   However, the nature of the divergence is such that convergence is possible as soon as 
the Biot-Savart law for the velocity $u = \mathcal{R}^\perp \theta$ is only slightly more regular. That motivated us to consider the natural family of active scalars that interpolates between the SQG and Euler equations, 
which are known as the modified or generalized surface quasi-geostrophic equations (mSQG). We refer the reader to Constantin et al. \cite{ciw}, Pierrehumbert et al. \cite{pierre}, Schorghofer \cite{schor} and 
Smith et al. \cite{smith} and references therein for a geophysical context for the mSQG equations.

In this paper we thus consider such inviscid modified surface quasi-geostrophic equation (mSQG) which is also an active scalar equation describing the transport of the scalar valued function 
$$\theta = \theta(x,t): \mathbb{T}^2 \times [0, \infty) \rightarrow \mathbb{R}$$ 
under the velocity field $u$, which itself is related to $\theta$ now via a regularized Biot-Savart law. More precisely, 
the mSQG equation that we consider in this paper is given as follows: 

\begin{equation}\label{msqg}
\begin{cases}
\theta_t + (u \cdot \nabla)\theta = 0, \; \; \; \; x \in \mathbb{T}^2 \mbox{ and } t >0\\
u = \mathcal{R}^\perp |D|^{-\delta} \theta,
\end{cases}
\end{equation}
where $\delta>0$. Here 
$|D|:=(-\Delta)^{1/2}$ and 
$\mathcal{R}^\perp := \nabla^\perp |D|^{-1}$ denotes the Riesz transform, where
 $\nabla^\perp = (-\partial_{x_2}, \partial_{x_1})$.

When $\delta = 1$ the equation \eqref{msqg} coincides with the 2D Euler equation with $\theta$ representing the vorticity $\omega$.
When $\delta =0$ the equation \eqref{msqg} coincides with the inviscid surface quasi geostrophic (SQG) equation. For $0< \delta < 1$ the relation between the velocity $u$ and the function $\theta$ is less singular than in the case
of the SQG equation \eqref{sqg}, but more singular than in the case of the 2D Euler equation \eqref{euler}.

The equation \eqref{msqg} in the regime
$0 < \delta < 1$ has been studied using deterministic tools by C\'ordoba, Fontelos, Mancho and Rodrigo \cite{cofomaro}, where the evolution of patch-like initial data has been considered, and by Chae, Constantin and Wu, \cite{chcowu} who established a regularity criterion. Furthermore, in recent work by Kiselev, Resnick, Yao and Zlatos \cite{kiselev1} and by Kiselev, Yao and Zlatos \cite{kiselev2} patch dynamics and local-in-time regularity on the whole plane and in the half-plane were studied and initial data leading to finite time singularity was exhibited. The mSQG equation \eqref{msqg} with more singular velocities $\delta < 0$ has been recently studied in \cite{chcocogawu}.

In this paper we use probabilistic tools to obtain a global flow in $H^s(\mathbb R^2), s < -3 +\delta $ for the mSQG equation \eqref{msqg} for any $0< \delta \leq 1$ via an invariant Gaussian measure $\rho$.  In particular, we implement the approach of \cite{alcr} in the context of 
mSQG as follows: 
\begin{enumerate} 
\item We work with the streamline formulation of the equation, whose (sufficiently) smooth solutions still conserve  `enstrophy' (for details see Section \ref{sec-main}).  We  then consider the infinite Gaussian measure $\rho$ constructed with respect to this enstrophy.
\item We rewrite the streamline formulation of the equation in terms of an orthonormal $L^2$ basis as an infinite ODE system, for which we analyze the coefficient corresponding to the nonlinear term, with the goal of obtaining an expectation estimate that will allow subsequent probabilistic analysis. 
\item We introduce an approximate system of ODE, which is still an infinite system, but which has truncated nonlinear term. We show that each of these systems has a global flow and leaves the Gaussian measure $\rho$ invariant. 
\item Finally we perform a probabilistic convergence argument, which will give us a global flow for the streamline formulation of the equation in the support of the Gaussian measure $\rho$. 
\end{enumerate} 
We refer the reader to Section \ref{sec-main} for the precise set up and statement of the Main Theorem as well as to Remark \ref{interpolate}.

Our paper can be understood as a probabilistic version in the context of the mSQG equations of the result of Resnick for SQG and
as a generalization of the work of Albeverio and Cruzeiro \cite{alcr} to velocity fields $u$ that are more singular ($0<\delta <1$) with respect to $\theta$ than in the case of Euler equations. As we alluded above, our analysis does not carry to the $\delta = 0$ case, which corresponds to the SQG equation. In particular, the expectation result of Proposition \ref{NonlinearityExpectation} fails in this case due
to a certain logarithmic divergence that appears already in each Fourier mode (c.f. Section \ref{subsec-expectation} ) and is hence independent of the choice of function space from which we consider the initial data.

\smallskip

Finally, we note that in \cite{alcr} Albeverio and Cruzeiro also 
considered the 2D Navier-Stokes equations, stochastically perturbed by a white noise. This line of work on constructing global weak solutions has been continued for stochastically perturbed Navier-Stokes equations using sophisticated tools from
probability in e.g. \cite{dade, alfe}, where certain types of uniqueness have been established too. 

The program of construction of probabilistic weak solutions introduced by Albeverio and Cruzeiro \cite{alcr} was recently implemented for the wave and dispersive equations by Burq, Thomann and Tzvetkov in \cite{btt}. 
Also we note that upon completion of this work we learned of the recent work of Symeonides \cite{sym}, who considered the averaged Euler equations in 2D and obtained a global flow in the support of a Gaussian measure constructed based on the associated enstrophy. 
Just as 2D Euler is to SQG, the averaged 2D Euler equation is one derivative smoother than the mSQG equation. 

\smallskip

{\bf{Outline of the paper.}} The problem is reformulated using stream functions in Section 2; Section 2 also reviews the standard construction of the Gaussian measure. In Section \ref{sec-prob} we recall the probabilistic tools used in the proof of our main result. Section 4 is devoted to checking the crucial fact that the expectation of the nonlinearity is finite whenever $\delta > 0$. We then introduce the approximate flows in Section 5, where we then show the invariance of the Gaussian measure under these flows. Using the tools of Section \ref{sec-prob}, we construct the candidate random flows in Section 6. The proof of the main result is given in Section 7.

\smallskip

{\bf{Acknowledgements.}} 
The authors express their gratitude to MSRI for the kind hospitality and stimulating environment during the Fall 2015 semester, where the project started while all four authors were in residence for their special jumbo program {\it New Challenges in PDE: Deterministic Dynamics and Randomness in High and Infinite Dimensional Systems.} The authors also thank IHES for their kind hospitality in Summer 2016 during of their {\it Ondes Non Lin\'eaires} program. Special thanks go also to Alessio Figalli, Martin Hairer, Luc Rey-Bellet, and Vlad Vicol for insightful and helpful discussions.

\section{The statement of the main result} \label{sec-main}

In this section we rewrite the mSQG equation in the streamline formulation, and we then review the standard construction the Gaussian invariant measure based on the conservation of the ``enstrophy" for solutions of the streamline formulation of the mSQG. Then we state the main result of this paper.

\subsection{Streamline formulation for the mSQG}
In an analogy with 2D Euler equations in the vorticity form, 
we introduce the streamline function $\varphi$ for our 
equation \eqref{msqg} so that 
we can write the velocity $u$ as 
\begin{equation} \label{vel} 
u = \nabla^\perp \varphi. 
\end{equation}
Having in mind that according to \eqref{msqg}
\begin{equation} \label{vel-orig} 
u = \mathcal{R}^\perp |D|^{-\delta} \theta = \nabla^\perp |D|^{-1} |D|^{-\delta} \theta 
\end{equation} 
such streamline function $\varphi$ is related to $\theta$ via
$$ \varphi = |D|^{-1-\delta}\theta$$
resulting in the streamline formulation of the mSQG equation \eqref{msqg}: 
\begin{equation}\label{mSQG-str}
\begin{cases}
(|D|^{1+\delta}\varphi)_t + (u \cdot \nabla)|D|^{1+\delta}\varphi = 0 \\
u = \nabla^\perp \varphi.
\end{cases}
\end{equation}

\begin{remark}\label{str-interpolation} 
We observe that by taking $\delta = 1$ in the streamline formulation \eqref{mSQG-str} we indeed recover
the known streamline formulation for 2D Euler equations:
\begin{equation}\label{2DE-str}
\begin{cases}
(\Delta \varphi)_t + (u \cdot \nabla) \Delta \varphi = 0 \\
u = \nabla^\perp \varphi
\end{cases}
\end{equation}
which was the starting point for the work \cite{alcr}. 
\end{remark} 

We find it convenient to rewrite the streamline formulation \eqref{mSQG-str} 
in terms of the regularized stream function $\psi$ introduced via
\begin{equation} \label{psi} 
|D|^{\delta} \varphi = \psi.
\end{equation} 
Then the  regularized streamline formulation that we work with reads as follows: 
\begin{equation}\label{msqg-our-str}
\begin{cases}
\psi_t + |D|^{-1}(u \cdot \nabla)|D|\psi = 0 \\
u = \nabla^\perp |D|^{-\delta} \psi.
\end{cases}
\end{equation}

Below we will abbreviate the nonlinearity in \eqref{msqg-our-str} by
\begin{equation}\label{our-str-B}
B(\psi, \psi) := -|D|^{-1}(\nabla^\perp |D|^{-\delta} \psi \cdot \nabla)|D|\psi
\end{equation}

We recall that for classical solutions to \eqref{msqg} the SQG enstrophy $\|\theta\|_{L^2}$ is conserved in time. We also note that, thanks to \eqref{psi}, we have that $\psi$ is mean zero. Consequently, the homogeneous and inhomogeneous Sobolev spaces restricted to our space of solutions are comparable. We therefore take the Sobolev norm 
$$\|f\|_{H^s(\mathbb{T}^2)}^2 := \sum_{k \in \mathbb{Z}^2} |k|^{2s}|\hat{f}(k)|^2,$$
and following  the notation in \cite{btt}, we introduce the spaces 
\begin{equation}\label{xspace}
X^\sigma := \bigcap_{s < \sigma} H^s.
\end{equation}

Finally, it follows from the definitions of $\psi$ and $\varphi$ and the conservation of $\|\theta \|_{L^2}$ for solutions $\theta$ of \eqref{msqg} that 
$\|\psi\|_{H^1}$ is formally conserved in time.  It is this conservation  of $\|\psi\|_{H^1}$ that  gives rise to the Gaussian measure $\rho$ introduced in the next subsection.

\subsection{The Gaussian Invariant Measure and its Support}

Here we review the construction of a centered Gaussian measure defined on functions $H^s$. The construction presented here is standard (see for example \cite{bog}) and we include it for the sake of completeness. 

If $\psi_k$ denote the Fourier coefficients of $\psi$, then heuristically we would like to define

\begin{equation}
``\, \, d\rho(\psi) \quad := \quad \frac{1}{Z} \prod_{k \in \mathbb{Z}^2} \exp(-2|k|^2|\psi_k|^2) d\psi_k \, \, "
\end{equation}
where $d\psi_k := dx_k \, dy_k$ is the Lebesgue measure on $\mathbb{C}$ associated to the variable $\psi_k = x_k + iy_k \in \mathbb{C}$, and $Z$ is the appropriate normalization factor needed to yield a probability measure. Unfortunately this heuristic expression is not well-defined.

To proceed rigorously, one constructs this measure as the weak limit of a sequence of premeasures defined on $H^s$ whose index will be determined later. In order to agree with the heuristically defined measure introduced above, fix the Hilbert space $\mathcal{H} = H^1$ corresponding to the conserved quantity of \eqref{msqg-our-str}, and introduce the correlation operator $$\mathcal{T} : \mathcal{H} \to \mathcal{H} : \mathcal{T}(\psi) = |D|^{2 - 2s} \psi.$$ The operator $\mathcal{T}$ has eigenvalues $$\lambda_k = |k|^{2 - 2s}$$ and corresponding eigenvectors $$e_k^s: = |k|^{-s}e^{ik \cdot x}$$ for $k\in \mathbb{Z}^2$. 
Note that $$\sum_{k \in \mathbb{Z}^2}^\infty |k|^2 |\psi_k|^2 = \langle \psi, \psi \rangle_{H^1} = \langle \mathcal{T}\psi, \psi \rangle_{H^s}.$$
This correlation operator is then used to build a sequence of pre-measures. Fix for the moment some $s \in \mathbb{R}$. In what follows, define for $k \in \mathbb{Z}^2$ the maximum norm $|k| = |(k_1, k_2)| = \max(k_1, k_2)$. For each $N \in \mathbb{N}$, define the projections 
$$\pi_N : H^s \to \mathbb{C}^{(2N + 1)^2}$$
by 
$$ \pi_N(\psi) = (\langle \psi, e_k^s \rangle_\mathcal{H})_{\{k\in \mathbb{Z}^2\, : |k |\leq N\}}$$ 
corresponding to the orthonormal basis $(e_k^s)$ of $H^s({\mathbb{T}}^2)$. We denote 
$$E_N := \text{span}\{e^s_k : |k| \leq N\}$$
We say that a set $M \subset H^s$ is \textit{$N$-cylindrical} if there exists some Borel set $F \subset \mathbb{C}^{(2N + 1)^2}$ for which $M = \pi_N^{-1}(F)$. 
Denote the algebra of $N$-cylindrical sets by $\mathcal{A}_N$. Similarly, call $M \subset \mathcal{H}$ \textit{cylindrical} if it is $N$-cylindrical for some $N \geq 1$, and introduce the algebra $\mathcal{A}$ of cylindrical sets.

Now define the following pre-measure for each $M = \pi_N^{-1}(F) \in \mathcal{A}$:
\begin{align}\label{premeasure}
\rho(M) & := \left(\prod_{|k|\leq N} \frac{1}{\sqrt{2\pi \lambda_k}} \right) \int_F \exp\left(-\frac12 \sum_{|k|\leq N} \lambda_k^{-1} |\psi_k|^2 \right) \, d\psi_1 \cdots d\psi_{(2N + 1)^2}.
\end{align}
This pre-measure is not necessarily countably additive. The following proposition (c.f. Proposition 1.3.1 of \cite{alcr}) 
gives us a criterion for when the pre-measure is countably additive:
\begin{proposition}\label{traceclass}
The Gaussian measure $\rho$ defined above is countably additive if and only if $\mathcal{T}^{-1}$ is of trace class. In this case, the minimal $\sigma$-algebra containing $\mathcal{A}$ is the Borel $\sigma$-algebra on $\mathcal{H}$.
\end{proposition}

In our case, we have that $\mathcal{T}^{-1}$ is of trace class provided
\begin{equation}
\sum_{k\in \mathbb{Z}^2}^\infty \lambda_k^{-1} =
\sum_{k \in \mathbb{Z}^2} |k|^{2s - 2} < \infty,
\end{equation}
which occurs provided we choose $s < 0$. Therefore the support of $\rho$ is in the space $X^0$.

Having constructed $\rho$, we adopt the usual notation for the expectation with respect to $\rho$:
$$\mathbb{E}_\rho(F(\psi)) = \int_{X^0} F(\psi) \, d\rho(\psi).$$
We note the following moment expectations which can be calculated explicitly from the definition \eqref{premeasure}:
\begin{align}\label{MomentExpectations}
\mathbb{E}_\rho(\psi_k) & = 0, \notag \\
\mathbb{E}_\rho(\psi_k \psi_{k^\prime}) & = 0, \\
\mathbb{E}_\rho(\psi_k \overline{\psi}_{k^\prime}) & = \frac{2\delta_{k, k^\prime}}{|k|^2|k^\prime|^2}. \notag
\end{align}
\smallskip

Given two Banach spaces $\mathfrak{X}, \mathfrak{Y}$ for which the support of $\rho$ is contained in $\mathfrak{X}$, we denote by $L^2_\rho(\mathfrak{X}, \mathfrak{Y})$ the space of all functions $\mathfrak{F} : \mathfrak{X} \to \mathfrak{Y}$ for which
$$\|\mathfrak{F}(\psi)\|_{L^2_\rho(\mathfrak{X}, \mathfrak{Y})}^2 := \int_{\mathfrak{X}} \|\mathfrak{F}(\psi)\|_\mathfrak{Y}^2 \, d\rho(\psi) < \infty.$$ 
Often in the sequel the domain $\mathfrak{X}$ will be understood from context, at which point we abbreviate $L^2_\rho(\mathfrak{X}, \mathfrak{Y}) = L^2_\rho(\mathfrak{Y})$.

\smallskip

It will be useful in the sequel to decompose $\rho$ along the perpendicular subspaces $H^s = E_N \oplus E_N^\perp$ for $s < 0$. Fixing some $N \in \mathbb{N}$, introduce the measure $\rho_N$ defined for $N$-cylindrical subsets $M$ of $H^s$ by
\begin{equation*}
\rho_N(M) := \left(\prod_{|k|\leq N} \frac{1}{\sqrt{2\pi \lambda_k}} \right) \int_{\pi_N(M)} \exp\left(-\frac12 \sum_{|k|\leq N} \lambda_k^{-1} \psi_k^2\right) \, d\psi_1 \cdots d\psi_{(2N + 1)^2}
\end{equation*}
For $N^\prime > N$, define $\pi_{N, N^\prime} : H^s \to \mathbb{C}^{(2N^\prime + 1)^2 - (2N + 1)^2}$ by 
$$\pi_{N, N^\prime}(\psi) = (\langle \psi, e_k^s \rangle)_{N < |k| \leq N^\prime}$$
Introduce the cylindrical measure $\rho_N^\perp$ defined on the set of cylindrical subsets $M^\prime = \pi_{N, N^\prime}(F)$ with $F \subset \pi_{N, N^\prime}(\mathcal{H})$ by
\begin{equation*}
\rho_N^\perp(M^\prime) = \left(\prod_{N < |k| \leq N^\prime} \frac{1}{\sqrt{2\pi \lambda_k}} \right) \int_{F} \exp\left(-\frac12 \sum_{N < |k| \leq N^\prime} \lambda_k^{-1} \psi_k^2 \right) \, d\psi_{(2N + 1)^2 + 1} \cdots d\psi_{(2N^\prime + 1)^2}
\end{equation*}
By Proposition \ref{traceclass}, the measure $\rho_N^\perp$ extends to a measure defined on $E_N^\perp$ supported on the same space $X^0$ as $\rho$. Moreover, if we decompose $\psi = \pi_N(\psi) + (\psi - \pi_N(\psi)) =: \psi_N + \psi_N^\perp$, we have by the Fubini-Tonelli Theorem that $d\rho(\psi) = d\rho_N(\psi_N) d\rho_N^\perp(\psi_N^\perp)$.

\subsection{Statement of the main result} 

With the construction of the Gaussian measure $\rho$ we can state our result that establishes global flows for the regularized streamline formulation \eqref{msqg-our-str}
of the (mSQG) equation. More precisely: 

\begin{theorem}\label{FlowOnArbLongFinite}
Let $T > 0$ be given. Then there exists a flow $\tilde{\Psi}(\omega, t)$ defined on a probability space $(\tilde{\Omega}, \tilde{\mathcal{F}}, \tilde{P})$ with values in $C([0, T] : X^{-2})$ such that for $\tilde{P}$-almost every $\omega \in \tilde{\Omega}$,
\begin{equation}\label{FullFlowDuHamel}
\tilde{\Psi}(\omega, t) = \tilde{\Psi}(0, \omega) + \int_0^t B(\tilde{\Psi}(\omega, \tau)) \, d\tau,
\end{equation}
where $B$ is as in \eqref{our-str-B}, as well as a Gaussian measure $\rho$ supported on $X^{-2}$ which is invariant with respect to $\tilde{\Psi}(t, \omega)$, i.e., for all measurable $F : X^{-2} \to \mathbb{R}$ and $t \in [0, T]$,
\begin{equation}\label{FullFlowInvariance}
\int_{\tilde{\Omega}} F(\tilde{\Psi}(\omega, t)) \, d\tilde{P}(\omega) = \int_{X^{-2}} F(\psi) \, d\rho(\psi).
\end{equation}
\end{theorem}

We will prove this theorem in Section 7. 

\smallskip

\begin{remark}\label{interpolate}
As mentioned in the introduction, the case $\delta = 1$ corresponds to the 2D Euler equations for which \cite{alcr} proves a similar result with global flows with values in $X^{-1}$.  It is therefore natural to wonder why in Theorem \ref{FlowOnArbLongFinite} the space $X^{-2}$ on which the flows take values is not instead the space $X^{-2 + \delta}$, improving as the smoothing parameter $\delta$ increases.  This apparent discrepancy is resolved by the fact that we construct global flows for the regularized stream function $\psi = |D|^{-\delta} \varphi$ rather than the streamline function $\varphi$ used in \cite{alcr}.  In order to directly compare our result with that of \cite{alcr} we must rewrite Theorem \ref{FlowOnArbLongFinite} in terms of $\varphi$ via \eqref{psi}. 
Doing so, our result implies the existence of global flows in the quantity $\varphi$ with values in the space $X^{-2 + \delta}, \, 0<\delta\leq 1$, thus recovering the probabilistic result of Albeverio-Cruzerio for Euler \cite{alcr}  when $\delta=1$.
\end{remark}

\smallskip

\begin{remark}\label{alternative}  If one takes the nonlinearity $B$ to be that of the equation \eqref{mSQG-str} for the streamline function instead of equation \eqref{msqg-our-str} and follows the argument of this paper\footnote{For the analogue of Subsection \ref{subsec-expectation} see the calculations in Appendix A.}, one can similarly construct global flows for $\varphi$ directly, based on the conservation in time of $\|\varphi\|_{H^{1+\delta}}$. Doing so however, offers no advantage over our approach here as explained in Remark \ref{interpolate} while complicating somewhat the exponents appearing in the calculations. 
\end{remark}

\begin{remark}
We briefly discuss the question of uniqueness of solutions to (mSQG). One expects due to the roughness of the solutions that uniqueness will be difficult to prove. We make no claim of uniqueness here, but instead mention some standard approaches which cannot be used in our setting. Existing deterministic local well-posedness results require too much regularity to be of use in our setting (c.f. \cite{comata1, coco}). An approach that recovers a weaker almost sure version of uniqueness can be found in the work of Ambrosio and Figalli \cite{AF}, where almost sure uniqueness of nonlinear flows of the form 
\begin{equation}\label{afflow}
X^\prime(t) = B(X(t), t)
\end{equation} is a consequence of uniqueness of the continuity equation 
\begin{equation}\label{conteq}
\frac{\partial \mu}{\partial t} + \text{div}_\rho(B\mu) = 0
\end{equation}
satisfied by the generalized flows $\mu$ associated to \eqref{afflow}\footnote{Here $\text{div}_\rho$ is the formal $L^2_\rho$-adjoint of the gradient.}. However, in \cite{AF}, uniqueness for \eqref{conteq} itself crucially depends on the nonlinearity $B$ taking values in the Cameron-Martin space associated to the Gaussian measure. In our case we show that $B$ takes values in a rougher space, which is not sufficient to allow the application of the result in \cite{AF}.
\end{remark}

\section{Probabilistic toolbox} \label{sec-prob} 

In this section we present a brief review of some classical probabilistic results on convergence of random variables. 
Throughout this section, for a metric space $S$ we denote by $\mathcal{B}(S)$ the Borel $\sigma$-algebra. 

\smallskip

We start by recalling the definitions of weak compactness and tightness, see e.g. \cite{KoSi-2007}, Section 8.3. 
\begin{definition} 
Let $S$ be a metric space. 
A family of probability measures $\{P_{\alpha} \}$ on $(S, \mathcal{B}(S))$
is said to be weakly compact if from any sequence $\{P_n \}_{n=1}^{\infty} \subset \{P_{\alpha} \}$
one can extract a weakly convergent subsequence $\{P_{n_k} \}_{k=1}^{\infty}$.
\end{definition}

\begin{definition} 
Let $S$ be a metric space. 
A family of probability measures $\{P_{\alpha} \}$ on $(S, \mathcal{B}(S))$
is said to be tight if for every $\epsilon > 0$ there exists a compact set 
$K_\epsilon \subset S$ such that 
$P(K_\epsilon) \geq 1 - \epsilon$ 
for each $P \in \{P_{\alpha} \}$. 
\end{definition} 

Now we are ready to state the compactness criterion of Prokhorov, see e.g. \cite{KoSi-2007}, Section 8.3 
or \cite{Bi}, Section 5, which we shall use in our Section 6. In particular, we shall use only the first part of Prokhorov theorem, but for completeness purposes we included the full statement. 

\begin{theorem} (Prokhorov) 
Suppose $S$ is a metric space. 
\begin{enumerate} 
\item[(i)] If a family of probability measures $\{P_{\alpha} \}$ on $(S, \mathcal{B}(S))$
is tight, then it is weakly compact. 

\item[(ii)] Suppose that $S$ is a separable complete metric space. If a family of probability measures $\{P_{\alpha} \}$ on $(S, \mathcal{B}(S))$
is weakly compact, then it is tight. 
\end{enumerate}
\end{theorem}

We conclude this short review with the statement of Skorohod's Theorem (for details see e.g. \cite{Bi} page 70), which we shall also use in Section 6 in order to construct our random flows. Given a probability space $(\Omega, \mathcal{F}, P)$ equipped with a probability measure, a measurable space $(E, \mathcal{E})$, and a random variable $X : (\Omega, \mathcal{F}, P) \to (E, \mathcal{E})$, the law $\mathcal{L}(X)$ of $X$ is the measure defined on the state space $(E, \mathcal{E})$ given by
\begin{equation}
\mathcal{L}(X)(A) = P(\{\omega \in \Omega : X(\omega) \in A\}) \qquad \text{for every }A \in \mathcal{E}.
\end{equation}

\begin{theorem} (Skorohod) 
Suppose $S$ is a separable metric space and $\{P_n \}_{n=1}^{\infty}$ and $P_\infty$ are probability measures on $(S, \mathcal{B}(S))$. 
If $P_n \rightarrow P_\infty$ weakly, then there exist
random variables $\{X_n\}_{n=1}^{\infty}$ and $X$ 
defined on a common probability space $(\Omega, \mathcal{F}, P)$ such that 
$\mathcal{L}(X_n) = P_n$, 
$\mathcal{L}(X) = P_\infty$, 
and 
$X_n \rightarrow X$ almost surely in $P$. 
\end{theorem}

\section{The streamline formulation as an infinite system of ODE}\label{expectation} 

In this section, we expand \eqref{msqg-our-str} explicitly into an infinite system of ODEs in the Fourier frequencies. This explicit representation is then used to show that the $H^s$-norm of the nonlinearity of the equation is finite in $L^2_\rho$ provided that $s < -2$. Informally, this calculation shows that $X^{-2}$ is the smallest space with respect to which the system \eqref{msqg-our-str} is closed in $\rho$-expectation. Moreover this estimate is a key ingredient of the compactness argument used to construct the random flows in our main result.

\subsection{An infinite system of ODE} 

Introducing an orthonormal basis $(e_k)$ of $L^2(\mathbb{T}^2)$, we write 
$$\psi= \sum_k \psi_k e_k.$$ 
Now we can write our modified streamline formulation \eqref{msqg-our-str} in terms of coefficients with respect to the orthonormal basis $(e_k)$ 
as follows: 
\begin{equation}\label{ode}
\frac{d \psi_k }{dt} = B_k(\psi), \\
\end{equation}
where $B_k$ denotes the coefficients of the nonlinearity $B$
\begin{equation}\label{our-str-B-again} 
B(\psi, \psi) := -|D|^{-1}(\nabla^\perp |D|^{-\delta} \psi \cdot \nabla)|D|\psi
\end{equation}
in this basis. 

We calculate the coefficients $B_k$ for $k\ne 0$ of the nonlinearity $B$ in this basis to be
\begin{align*}
B_k & = \sum_{h + h^\prime = k, \, h, h'\ne 0} -|k|^{-1} |h|^{-\delta} h^\perp \cdot h^\prime |h^\prime| \psi_h \psi_{h^\prime} \\
& = \frac12 \biggl( \sum_{h + h^\prime = k, \, h, h'\ne 0} -|k|^{-1} |h^\prime|^{-\delta} (h^\prime)^\perp \cdot h |h| \psi_h \psi_{h^\prime} + \sum_{h + h^\prime = k, \, h, h'\ne 0} -|k|^{-1} |h|^{-\delta} h^\perp \cdot h^\prime |h^\prime| \psi_h \psi_{h^\prime} \biggr) \\
& = \frac12 \biggl( \sum_{h + h^\prime = k, \, h, h'\ne 0} |k|^{-1} |h^\prime|^{-\delta} h^\perp \cdot h^\prime |h| \psi_h \psi_{h^\prime} + \sum_{h + h^\prime = k, \, h, h'\ne 0} -|k|^{-1} |h|^{-\delta} h^\perp \cdot h^\prime |h^\prime| \psi_h \psi_{h^\prime} \biggr) \\
& = \frac12 \sum_{h + h^\prime = k, \, h, h'\ne 0} |k|^{-1} (h^\perp \cdot h^\prime) (|h^\prime|^{-\delta} |h| - |h|^{-\delta} |h^\prime| ) \psi_h \psi_{h^\prime} 
\end{align*}
where in the above we symmetrized the sum using the divergence-free structure with an eye towards minimizing the number of positive factors of $|h|$ and $|h^\prime|$. This gives
\begin{equation}\label{BDefinition}
B_k = -\frac12 \sum_{h\ne 0,\, k} \left(h^\perp \cdot \frac{k}{|k|}\right) (|k - h|^{-\delta} |h| - |h|^{-\delta} |k - h| ) \psi_h \psi_{k - h} = :\sum_{h\ne 0,\, k} \alpha_{k, h} \psi_h \psi_{k - h}
\end{equation}
Notice that one can readily check that $\alpha_{k, h} = \alpha_{k, k - h}$ for all $k, h \in \mathbb{Z}^2$.

\subsection{Expectation of the Nonlinear Term} \label{subsec-expectation} 

The subsequent analysis depends strongly on the following crucial proposition:

\begin{proposition}\label{NonlinearityExpectation}
Let $0<\delta \leq 1$. Then $B$ given by \eqref{our-str-B-again} satisfies
$$B \in L^2_\rho(H^s, H^s), \; \; \; \; \mbox{ for all } s < -2.$$ 
\end{proposition}

\begin{proof}
We need to show that expectation of the expression $\|B(\psi)\|_{H^s}^2$ is finite. Using the expectations of the moments in \eqref{MomentExpectations} as well as the fact that $\alpha_{k, h} = \alpha_{k, k - h}$, we first have that

\begin{align*}
\mathbb{E}_\rho(\|B\|_{ H^s}^2) & = \sum_{k\ne 0} |k|^{2s} \sum_{h, \, h^\prime\ne 0} \alpha_{k, h} \alpha_{k, h^\prime} \mathbb{E}(\psi_h \psi_{k - h} \bar{\psi}_{h^\prime} \bar{\psi}_{k - h^\prime}) \\
& = 2 \sum_{k \ne 0} |k|^{2s} \sum_{h, h^\prime\ne 0} \frac{\alpha_{k, h} \alpha_{k, h^\prime}}{|h|^2|h - k|^2} (\delta_{h, h^\prime} + \delta_{h, k - h^\prime}) \\
& = 4 \sum_{k \ne 0} |k|^{2s} \sum_{h, h^\prime\ne 0} \frac{\alpha_{k, h}^2}{|h|^2|h - k|^2}.
\end{align*}
We focus first on establishing that the inner sum in this last expression converges\footnote{This is precisely the step that fails in the classical inviscid SQG model with $\delta = 0$.}. Substituting our expression for $\alpha_{k, h}$, we have
\begin{align*}
& \quad \sum_{h\ne 0, k} \frac{\alpha_{k, h}^2}{|h|^2|h - k|^2} \\
& = \frac14 \sum_{h\ne 0,k} \frac{\left(h^\perp \cdot \frac{k}{|k|}\right)^2 (|k - h|^{-\delta} |h| - |h|^{-\delta} |k - h| )^2}{|h|^2|h - k|^2} \\
& \lesssim \sum_{h\ne 0, k} \frac{(|k - h|^{-\delta} |h| - |h|^{-\delta} |k - h| )^2}{|h - k|^2} \\
& = \sum_{h\ne 0, k} \frac{\Bigl(|h|^{-\delta}(|h| - |k - h|) + |h|(|k - h|^{-\delta} - |h|^{-\delta}) \Bigr)^2}{|h - k|^2} \\
& = \sum_{h\ne 0,k} \frac{\Bigl(|h - k|^{-\delta}(|h| - |k - h|) + (|h|^{-\delta} - |h - k|^{-\delta})(|h| - |k - h|) + |h|(|k - h|^{-\delta} - |h|^{-\delta}) \Bigr)^2}{|h - k|^2} \\
& \lesssim \sum_{h \neq 0,k} \frac{|h - k|^{-2\delta}(|h| - |k - h|)^2}{|h - k|^2} + \sum_{h \neq 0,k} \frac{(|h|^{-\delta} - |h - k|^{-\delta})^2(|h| - |k - h|)^2}{|h - k|^2} +\sum_{h \neq 0,k} \frac{ |h|^2(|k - h|^{-\delta} - |h|^{-\delta})^2}{|h - k|^2} \\
& := S_1 + S_2 + S_3,
\end{align*}
where we have decomposed the sum into three sums based on the three terms in the numerator of the summand. Before estimating in detail, we present

\begin{lemma}\label{DeltaDifference}
Assume that $|h| \geq 2|k|$. Then $\left| |k - h|^{-\delta} - |h|^{-\delta}\right| \leq \delta |k| |h|^{-1-\delta}$.
\end{lemma}

\begin{proof}
Note that $|h| \geq 2|k|$ implies that $\frac12 |h| \leq |k - h| \leq \frac32 |h|$. Then
\begin{align*}
\left| |k - h|^{-\delta} - |h|^{-\delta}\right| & = \left| \int_{|h|}^{|k - h|} \delta z^{-1 - \delta} \, dz \right| \\
& \leq \left| |k - h| - |h| \right| \delta \max_{\lambda \in [0, 1]} \Bigl(\lambda|k - h| + (1 - \lambda)|h|\Bigr)^{-\delta - 1} \\
& \leq C_\delta |k| |h|^{-1-\delta},
\end{align*}
as $\delta > 0$.
\end{proof}

We return to estimating the above sums. In estimates on $S_1$ and $S_2$ we utilize the immediate consequence of the triangle inequality: 
\begin{equation}\label{expect-triangle} 
\left| \, |h| - |k-h| \,\right| \leq |k|.
\end{equation}
Then, for any $\delta > 0$,  the sum $S_1$ can be bounded from above as follows: 
\begin{align*}
S_1 & \leq \sum_{h\ne 0, k} \frac{|k|^2}{|k - h|^{2 + 2\delta}} \lesssim |k|^2.
\end{align*}
Next, utilizing \eqref{expect-triangle} and decomposing $S_2 = S_2^{lo} + S_2^{hi}$ depending on whether $|h|$ is less or greater than $2|k|$ respectively, and using 
Lemma \ref{DeltaDifference} we have that
\begin{align*}
S_2 & = S_2^{lo} + S_2^{hi} \\
& \lesssim \sum_{|h| \leq 2|k|,\, h\ne 0, k} \frac{|k|^2}{|h|^{2\delta}|k - h|^2} + \sum_{|h| \leq 2|k|\, h\ne 0, k} \frac{|k|^2}{|k - h|^{2+2\delta}} +\sum_{|h| \geq 2|k|, k\neq 0} \frac{\delta^2|k|^4 |h|^{ -2- 2\delta}}{|h - k|^2} \\
& \lesssim |k|^2 +|k|^{2-2\delta}.
\end{align*}
Similarly, by applying Lemma \ref{DeltaDifference}, we have 
\begin{align*}
S_3 & = S_3^{lo} + S_3^{hi} \\
& \lesssim \sum_{|h| \leq 2|k|,\, h\ne 0, k} \frac{|h|^2}{|h|^{2\delta}|k - h|^2} + \sum_{|h| \leq 2|k|\, h\ne 0, k} \frac{|h|^2}{|k - h|^{2+2\delta}} + \sum_{|h| \geq 2|k|,\, k\ne 0} \frac{|h|^2\delta^2|k|^2 |h|^{-2-2\delta}}{|h - k|^2} \\
& \lesssim |k|^2 + |k|^{2-2\delta}. 
\end{align*}
The maximum amount of smoothness imposed on $k$ from evaluating these sums is comparable to $|k|^2$. Therefore, the expectation at the beginning of the calculation can be estimated by
\begin{equation}\label{expectation}
\mathbb{E}_\rho(\|B\|_{H^s}^2) \lesssim \sum_k |k|^{2s + 2},
\end{equation}
which is finite provided we choose $s < -2$.
\end{proof}

By repeating the argument above that gives the finiteness of the expectation\footnote{Using sums ranging over frequencies $N_1 \leq |h|, |k - h| \leq |N_2|$.}
we establish a crucial convergence result, that relates the full nonlinearity and its truncated version appearing in Galerkin approximations of \eqref{msqg-our-str} (which will be 
introduced and analyzed in Section \ref{sec-Gal}). 
In order to state the convergence result we introduce the projection onto the subspace spanned by $(e_k)_{|k|\leq N}$ and denote it by $\Pi_N$, 
as well as the orthogonal projection $\Pi_N^\perp := (I - \Pi_N)$. Now we are ready to 
introduce the truncated version of the nonlinearity via: 
\begin{equation} \label{BN} 
B^N(\psi) := \Pi_N B( \Pi_N \psi).
\end{equation} 

The convergence results can be stated as follows: 

\begin{proposition}\label{SumTailGoesToZero}
If $s < -2$ then $B^N \to B \mbox{ in } L^2_\rho(H^s, H^s).$
\end{proposition}

\section{Construction and Invariance of the Truncated mSQG Flows} \label{sec-Gal}

We will construct our eventual random flows from a sequence of flows satisfying a truncated version of \eqref{msqg-our-str}. The dynamics of these approximate flows are only nontrivial on finite dimensional subspaces, leave $\rho$ invariant and conserve the $H^1$-norm. We show in this section that these properties suffice to construct flows for the approximate systems with arbitrarily long lifespans.

Since $H^s$ for $s < -2$ is the natural space in which to consider the nonlinearity $B$ in expectation, from this point we regard $\rho$ as defined on $X^{-2}$.

We introduce the $N$th approximate flow $\Psi^N(t, \psi)$ as the solution of the Cauchy problem 
\begin{equation}\label{FiniteDimFlow}
\begin{cases}
\partial_t \Psi^N(t) = B^N(\Psi^N(t)) \\
\Psi^N(0, \psi) = \psi.
\end{cases}
\end{equation}
If we let $V^N$ satisfy the finite dimensional system
\begin{equation}\label{FiniteDimFlowV}
\begin{cases}
\partial_t V^N(t) = B^N(V^N) \\
V^N(0, \psi) = \Pi_N \psi,
\end{cases}
\end{equation}
then observe that the flow $\Psi^N$ can be decomposed into 
\begin{equation}\label{FiniteDimFlowDecomp}
\Psi^N(t, \psi) = V^N(t, \Pi_N \psi) + \Pi_N^\perp \psi.
\end{equation}

Denote the $e_k$-component of $\Psi^N$ by $\Psi^N_k$. 
\begin{lemma}
Let $T > 0$ be given. Then there exists a unique flow $\Psi^N(t, \psi)$ solving \eqref{FiniteDimFlow} for all $t \in [0, T]$ with $\Psi^N_k(t, \psi) \in C([0, T], \mathbb{C})$ and which leaves the measure $\rho$ invariant.
\end{lemma}

\begin{proof}
Let $\Psi^N(t, \psi)$ solve \eqref{FiniteDimFlow}, and consider $V^N$ as defined in \eqref{FiniteDimFlowV}. In order to check that the $H^1$-norm of $\Psi^N(t, \psi)$ is conserved, it suffices by \eqref{FiniteDimFlowDecomp} it suffices to check that that the $H^1$ norm of $V^N$ is conserved. By definition $$B_N(V^N) = -\Pi_N|D|^{-1}(\nabla^\perp |D|^{-\delta} V^N \cdot \nabla)\Pi_N |D| V^N,$$ and so
\begin{align}
\frac12 \frac{d}{dt} \|\,|D|V^N\|_{L^2}^2 & = \langle |D|V^N, -\Pi_N(\nabla^\perp |D|^{-\delta} V^N \cdot \nabla)|D| \Pi_N V^N \rangle \nonumber \\
& = \langle \Pi_N |D| V^N, -(\nabla^\perp |D|^{-\delta} V^N \cdot \nabla)\Pi_N |D|V^N \rangle \label{Lemma5-1Step3} \\ 
& = \langle \Pi_N |D| V^N, (\Pi_N |D|V^N) (\nabla \cdot \nabla^\perp) |D|^{-\delta} V^N \rangle \nonumber \\
& \quad + \langle (\nabla^\perp |D|^{-\delta} V^N \cdot \nabla) \Pi_N |D| V^N, \Pi_N |D|V^N \rangle \nonumber\\
& = \langle \Pi_N |D| V^N, (\nabla^\perp |D|^{-\delta} V^N \cdot \nabla)\Pi_N |D|V^N \rangle \label{Lemma5-1Step5} \\ 
& = 0, \label{Lemma5.1Conclusion}
\end{align}
where \eqref{Lemma5.1Conclusion} follows since \eqref{Lemma5-1Step3} and \eqref{Lemma5-1Step5} in the above chain of equalities are exactly opposites of each other. Therefore, $\|V^N\|_{H^1}^2$ is conserved in time. Local existence of the flow $V^N$ follows by the classical Picard-Lindel\"of Theorem, and then global existence of the flow follows by the uniform boundedness of the $H^1$ norm of $V^N$ in time.

Next, we claim that the flow $V^N$ preserves the finite dimensional Lebesgue measure. By Liouville's Theorem, it suffices to check that the divergence of $B^N(V^N)$ is zero. Denoting the coordinates of $V^N$ in $E_N$ by $(v_1, v_2, \ldots, v_N)$, we have
\begin{align*}
\text{div}_{E_N}(B^N(V^N)) & = \text{div}_{E_N} \left( \sum_{|h| \leq N} \alpha_{h, k} v_h v_{k - h} \right) \\
& = \sum_{|k| \leq N} \frac{\partial}{\partial v_k} \sum_{|h| \leq N} \alpha_{h, k} v_h v_{k - h} \\
& = \sum_{|k| \leq N} (\alpha_{k, k} + \alpha_{0, k}) v_0 \\
& = 0,
\end{align*}
where the last inequality follows immediately by inspection of the formula \eqref{BDefinition} for $\alpha_{h, k}$.

Since the flow $V^N$ both conserves the $H^1$ norm and leaves Lebesgue measure invariant, it also leaves the finite dimensional Gaussian measure $\rho_N \circ \pi_N^{-1}$ invariant. Then, writing 
$$\psi = \Pi_N \psi + \Pi_N^\perp \psi =: \psi_N + \psi_N^\perp$$ 
along the orthogonal decomposition $H^s = E_N \oplus E_N^\perp$, we have by Fubini-Tonelli that for any $F : H^s \to \mathbb{R}$ with $s < 0$,
\begin{align*}
\int_{H^s} F(\Psi^N(t, \psi)) \, d\rho(\psi) & = \int_{E_N^\perp} \left( \int_{E_N} F\left(V^N(t, \psi_N) + \psi_N^\perp\right) d\rho_N(\psi_N) \right) d\rho_N^\perp(\psi_N^\perp) \\
& = \int_{E_N^\perp} \left( \int_{E_N} F\left(\psi_N + \psi_N^\perp \right) d\rho_N(\psi_N) \right) d\rho_N^\perp(\psi_N^\perp) \\
& = \int_{H^s} F(\psi) \, d\rho(\psi)
\end{align*}
where we applied the invariance of $\rho_N(\psi_N)$ under $V^N(t, \psi_N)$ with the measurable function $F(\cdot + \psi_N^\perp)$. Finally, since every $\psi \in X^\sigma$ is in some $H^s$ for $s < \sigma$, the invariance also holds in any space $X^\sigma$ with $\sigma \leq 0$ and measurable $F : X^\sigma \to \mathbb{R}$.
\end{proof}

\section{Convergence argument}\label{ConvergenceArgument}

From this point onward, we consider only a Gaussian measure $\rho$ constructed as in Section 2.2 defined on $X^{-2}$.

In order to construct random flows from our (essentially) finite-dimensional deterministic flows, we regard the deterministic flows $\Psi^N(t, \psi)$ as stochastic processes sampled from $X^s$ with state space $C([0, T] : X^s)$ and introduce the measures $\nu_N$ supported on the infinite dimensional path space $C([0, T] : H^s)$ as their laws:
\begin{equation}\label{NuNDefn}
\nu_N(\Gamma) = \rho(\{\psi \in X^{-2} : \Psi^N(\psi, \cdot) \in \Gamma\}), \qquad \Gamma \subset C([0, T] : X^{-2}). 
\end{equation}

Our first goal is to show that the laws $\nu_N$ can be used to construct a measure $\nu$ that will serve as the law of our eventual candidate flows. We accomplish this using the compactness provided by Prokorov's Lemma; to verify the hypotheses of that lemma we first need to show some useful analytic estimates.

\begin{lemma}\label{GagliardoNirenburg}
Let $T > 0$ be given. Let $-\infty < s_2 \leq s_1 < +\infty$, and denote $\overline{s} = \frac12(s_1 + s_2)$. Suppose that $\gamma \in L^2_T H^{s_1}$ and $\partial_t \gamma \in L^2_T H^{s_2}$. Then for all $s < \overline{s}$, we have $\gamma \in L^\infty_T H^{\overline{s}}$ and
\begin{equation}
\|\gamma\|_{L^\infty_T H^{\overline{s}}} \lesssim \|\gamma\|_{L^2_T \dot{H}^{s_1}}^\frac12 \|\gamma\|_{H^1_T H^{s_2}}^\frac12
\end{equation}
\end{lemma}

\begin{proof}
By a paradifferential version of the classical Gagliardo-Nirenberg inequality, c.f. Lemma 3.3 of \cite{btt}.
\end{proof}

Let $X$ be a Banach space containing the support of some measure $\mu$. In what follows we abuse notation slightly by introducing the abbreviated notation
\begin{equation}
\|f\|_{L^2_\mu X}^2 := \int_X \|f\|_X^2 \, d\mu(f)
\end{equation}

\begin{lemma}\label{LoControl}
Let $T > 0$ and $\sigma < -2$ be given. Then for any $\gamma \in L^2_T H^{\sigma}$ we have
\begin{equation}
\|\gamma\|_{L^2_{\nu_N} L^2_T H^\sigma} \lesssim T.
\end{equation}
\end{lemma}

\begin{proof}
We calculate that
\begin{align}
\|\gamma\|_{L^2_{\nu_N} L^2_T H^\sigma}^2 
& := \int_{C([0, T] : X^{-2})} \|\gamma\|_{L^2_T H^\sigma}^2 \, d\nu_N(\gamma) \nonumber \\
& = \int_{C([0, T] : X^{-2})} \int_0^T \|\gamma(\tau)\|_{H^\sigma}^2 \, d\tau \, d\nu_N(\gamma) \nonumber \\
& = \int_0^T\int_{X^{-2}} \|\Psi^N(\cdot, \psi)\|_{H^\sigma}^2 \, d\rho(\psi) \, d\tau \label{Lolaw} \\
& = T\int_{X^{-2}} \|\psi\|_{H^\sigma}^2 \, d\rho(\psi) \label{LoInvariance} \\
& \lesssim T.\nonumber
\end{align}
where to obtain \eqref{Lolaw} we used the Fubini-Tonelli Theorem along with the definition of the law \eqref{NuNDefn}, 
and to obtain \eqref{LoInvariance} we used the invariance of the flow $t \mapsto \Psi^N(t, \psi)$ with respect to the measure $\rho$.
\end{proof}

\begin{lemma}\label{HiControl}
Let $T > 0$ and $\sigma < -2$ be given. Then for any $\gamma$ such that $\partial_t \gamma \in L^2_T H^{\sigma}$
\begin{equation}
\|\partial_t \gamma\|_{L^2_{\nu_N}L^2_T H^\sigma}^2 \lesssim T.
\end{equation}
\end{lemma}

\begin{proof}
\begin{align}
\|\partial_t \gamma\|_{L^2_{\nu_N}L^2_T H^\sigma}^2 
& := \int_{C([0, T] : X^{-2})} \int_0^T \|\partial_t \gamma(\tau)\|_{H^\sigma}^2 \, d\tau \, d\nu_N(\gamma) \nonumber \\
& = \int_0^T \int_{C([0, T] : X^{-2})} \|\partial_t \gamma(\tau)\|_{H^\sigma}^2 \, d\nu_N(\gamma) \, d\tau \label{HiFubini} \\
& = \int_0^T \int_{X^{-2}} \|\partial_t \Psi^N(\tau, \psi)\|_{H^\sigma}^2 \, d\rho(\psi) \, d\tau \label{HiNuNDefn}\\
& = \int_0^T \int_{X^{-2}} \|B^N(\Pi_N \Psi^N(\tau,\psi))\|_{H^\sigma}^2 \, d\rho(\psi) \, d\tau \label{HiFiniteDimFlow} \\
& = \int_0^T \int_{X^{-2}} \|B^N(\psi)\|_{H^\sigma}^2 \, d\rho(\psi) \, d\tau \label{HiInvariance} \\
& \lesssim T, \label{HiBExp}
\end{align}
where in \eqref{HiFubini} we used the Fubini-Tonelli Theorem, to obtain \eqref{HiNuNDefn} we used the definition of the law \eqref{NuNDefn}, 
to obtain \eqref{HiFiniteDimFlow} we used \eqref{FiniteDimFlow}, 
to obtain \eqref{HiInvariance} we used the invariance of $\rho$ under $\Psi^N$, 
and to obtain \eqref{HiBExp} we crucially used Proposition \ref{NonlinearityExpectation}.
\end{proof}

\begin{proposition}\label{Tightness}
Let $T > 0$ and $s < -2$ be given. Then the family $(\nu_N)$ of measures is tight on $C([0, T], H^s)$.
\end{proposition}

\begin{proof}
Introduce the H\"older space $C^\frac12([0, T], H^\sigma) =:C_T^\frac12 H^\sigma$ with norm
\begin{equation}
\|\gamma\|_{C^\frac12_T H^\sigma} := \|\gamma\|_{L^\infty_T H^\sigma} + \sup_{t_1 \neq t_2 \in [0, T]} \frac{\|\gamma(t_1) - \gamma(t_2)\|_{H^\sigma}}{|t_1 - t_2|^\frac12}.
\end{equation}
For $s < -2$ given, choose $-\infty < s_2 \leq s_1 < +\infty$ so that $s_1 < -2$ and $s < \overline{s} := \frac12(s_1 + s_2)$. By Lemma \ref{GagliardoNirenburg}, we have the estimate
\begin{align*}
\|\gamma\|_{C^\frac12_T H^{\overline{s}}} & \lesssim \|\gamma\|_{L^2_T H^{s_1}} + \|\gamma\|_{L^2_T H^{s_2}} + \|\partial_t \gamma\|_{L^2_T H^{s_2}} + \sup_{t_1 \neq t_2 \in [0, T]} \frac{\|\gamma(t_1) - \gamma(t_2)\|_{H^{\overline{s}}}}{|t_1 - t_2|^\frac12} \\
& := M_1 + M_2 + M_3 + M_4.
\end{align*}
By Lemma \ref{LoControl}, we have $\|M_1\|_{L^2_{\nu_N}} + \|M_2\|_{L^2_{\nu_N}} \lesssim T$, and by Lemma \ref{HiControl} we have $\|M_3\|_{L^2_{\nu_N}} \lesssim T$. We also have using H\"older's inequality that
\begin{align*}
M_4 & = \frac{1}{|t_1 - t_2|^\frac12} \left\| \int_{t_1}^{t_2} \partial_t \gamma(\tau) \, d\tau \right\|_{H^{\overline{s}}} \\
& \leq \frac{1}{|t_1 - t_2|^\frac12} \int_{t_1}^{t_2} \|\partial_t \gamma(\tau)\|_{H^{\overline{s}}} \, d\tau \\
& \leq \|\partial_t \gamma\|_{L^2_T H^{\overline{s}}}
\end{align*}
so that $\|M_4\|_{L^2_{\nu_N}} \lesssim T$ as well, by Lemma \ref{HiControl}. To sum, we obtain
\begin{equation}\label{gammabound}\|\gamma\|_{C^\frac12_T H^{\overline{s}}}\lesssim T.
\end{equation}
We now construct the compact exhaustion of sets required to show tightness: for $\delta > 0$, define
\begin{equation}
K_\delta = \{\gamma \in C^0_T H^s : \|\gamma\|_{C^\frac12_T H^{\overline{s}}} \leq \delta^{-1}\}.
\end{equation}
Since the inclusion $C^\frac12_T H^{\overline{s}} \subset C^0_T H^s$ is compact, $K_\delta$ is also compact. But by Chebychev's inequality and \eqref{gammabound} we have
\begin{equation}
\nu_N(K_\delta^c) \leq \delta^2 \|\gamma\|_{L^2_{\nu_N} C^\frac12_T H^{\overline{s}}}^2 \lesssim \delta^2 T^2,
\end{equation}
which demonstrates that the family $\{\nu_N : N \in \mathbb{N}\}$ is tight in $C([0, T], H^s)$.
\end{proof}

For any fixed $s < -2$, Prokhorov's Lemma now implies the existence of a subsequence of measures which converges weakly to another measure $\nu_s$. Since Proposition \ref{Tightness} holds for arbitrary $s < -2$, a standard diagonalization argument allows us to select another subsequence (for which we abuse notation in denoting it by $(\nu_N)$) which converges to the measure $\nu$ supported on $C([0, T], X^{-2})$. Then, Skorokhod's Lemma assures the existence of a probability space $(\tilde\Omega, \tilde{\mathcal{F}}, \tilde P)$ as well as random processes $\tilde{\Psi}^N(\omega)$ and $\tilde{\Psi}(\omega)$ with values in $C([0, T], X^{-2})$ whose laws are $\nu_N$ and $\nu$ respectively, and moreover so that $$\tilde{\Psi}^N(\omega) \to \tilde{\Psi}(\omega) \text{ in } C([0, T], X^s), \qquad \text{a.e. }\omega \in \tilde\Omega.$$ 
Observe that by construction the laws of $\tilde{\Psi}^N$ and that of $\Psi^N$ in the path space $C([0, T], X^s)$ are the same. 

\section{Proof of Theorem \ref{FlowOnArbLongFinite}}

With the construction of $\tilde{\Psi}$ given in Section 6, we may now give the

\bigskip 
\begin{proof}[Proof of Theorem \ref{FlowOnArbLongFinite}] We first claim that $\rho$ agrees with the measures $\nu^{(t)}$ conditioned on evaluation at a fixed time $t \in [0, T]$, defined as the law of $\tilde{\Psi}(t, \omega)$ for fixed $t \in [0, T]$. To see this, construct for fixed $t$ and $A \subset X^{-2}$ the subset $$\Gamma_{A, t} := \{\gamma \in C([0, T] : X^{-2}) : \gamma(t) \in A\}.$$ Then
\begin{align}
\nu_N^{(t)}(A) & := \nu_N(\Gamma_{A, t}) \\
& = \rho(\{\psi \in X^{-2} : \Psi^N(\psi) \in \Gamma_{A, t} \}) \label{NuTNuNDefn} \\
& = \rho(\{\psi \in X^{-2} : \Psi^N(\psi, t) \in A \}) \\
& = \rho(\{\psi \in X^{-2} : \psi \in A\}) \label{NuTInvariance} \\
& = \rho(A),
\end{align}
where in \eqref{NuTNuNDefn} we used the definition of the law \eqref{NuNDefn}, and in \eqref{NuTInvariance} we used the invariance of $\rho$ under the flow $\Psi^N$ in \eqref{FiniteDimFlow}. Therefore $\nu_N^{(t)} = \rho$, so that $\rho$ is the distribution of $\tilde{\Psi}^N(\omega, t)$ for each $t \in [0, T]$. Now we may identify the time-evaluated measure $\nu^{(t)}$ to be
\begin{align*}
\nu^{(t)}(A) = \lim_{N \to \infty} \nu^{(t)}_N(A) = \rho(A).
\end{align*}
Thus $\rho$ is an invariant measure for the random flow $\tilde{\Psi}$; in particular for every measurable $F : X^{-2} \to \mathbb{R}$, we have shown \eqref{FullFlowInvariance}:
\begin{equation}
\int_{\tilde{\Omega}} F(\tilde{\Psi}(\omega,t)) \, d\tilde{P}(\omega) = \int_{X^{-2}} F(\psi) \, d\rho(\psi)
\end{equation}
This in turn implies that $\tilde{\Psi}$ takes values in $X^{-2}$ almost surely in $\tilde{\Omega}$, upon choosing $F(\cdot) = \|\cdot\|_{H^\sigma}$ for each $\sigma < -2$ above. Therefore we can expand 
$$\tilde{\Psi}(\omega, x, t) =: \sum_{k \in \mathbb{Z}^2_0} \tilde{\Psi}_{k}(\omega, t) e_{k}(x)$$ as usual.

Next we verify \eqref{FullFlowDuHamel}. Recall that for each $N$ we have from \eqref{FiniteDimFlow} that
\begin{equation}
\Psi^N(t, \psi) = \Psi^N(0, \psi) + \int_0^t B^N(\Psi^N(\tau, \psi)) \, d\tau
\end{equation}
Consider the residual associated to this equation
\begin{equation}
Y^N(t, \psi) := \Psi^N(t, \psi) - \Psi^N(0, \psi) - \int_0^t B^N(\Psi^N(\tau, \psi)) \, d\tau
\end{equation}
as well as the corresponding residual for the random variable $\tilde{\Psi}^N$:
\begin{equation}
\tilde{Y}^N(t, \omega) := \tilde{\Psi}^N(t, \omega) - \tilde{\Psi}^N(0, \omega) - \int_0^t B^N(\tilde{\Psi}^N(\tau, \omega)) \, d\tau
\end{equation}
Since $\mathcal{L}(\Psi^N) = \mathcal{L}(\tilde{\Psi}^N)$ by Skorokhod's Lemma, we have $\mathcal{L}(Y^N) = \mathcal{L}(\tilde{Y}^N)$. However, since $\Psi^N$ is a solution of \eqref{FiniteDimFlow}, we have that $\mathcal{L}(\tilde{Y}^N) = \mathcal{L}(Y^N) = \delta_0$. This implies that $\tilde{Y}^N = 0$ almost surely in $\tilde{\Omega}$, and so we have the following almost everywhere pointwise equation for $\tilde{\Psi}^N$:
\begin{equation}
\tilde{\Psi}^N(t, \omega) = \tilde{\Psi}^N(0, \omega) + \int_0^t B^N(\tilde{\Psi}^N(\tau, \omega)) \, d\tau, \qquad \text{a.e. } \omega \in \tilde{\Omega}.
\end{equation}
By the construction using Skorokhod's Lemma, we already have that $\tilde{\Psi}^N(t, \omega) \to \tilde{\Psi}(t, \omega)$ and $\tilde{\Psi}^N(0, \omega) \to \tilde{\Psi}(0, \omega)$ for each $t \in [0, T]$, a.s. in $\tilde{\Omega}$. Hence to show $\tilde{P}$-a.e. convergence of the Duhamel term it suffices to show (possibly up to the extraction of another subsequence) that
\begin{equation}\label{DuHamelConverges?}
\int_{\tilde{\Omega}} \left| \int_0^t B^N(\tilde{\Psi}^N(\omega, \tau) \, d\tau - \int_0^t B(\tilde{\Psi}(\omega, \tau)) \, d\tau \right| \, d\tilde{P}(\omega)
\end{equation}
approaches zero as $N \to \infty$. Following \cite{dade}, introduce an auxiliary index $M \in \mathbb{N}$; then we expand \eqref{DuHamelConverges?} as
\begin{align*}
\lim_{N \to \infty} \eqref{DuHamelConverges?} & \leq \lim_{N \to \infty} \int_{\tilde{\Omega}} \int_0^t |B^N(\tilde{\Psi}^N(\omega, \tau)) - B(\tilde{\Psi}^N(\omega, \tau))| \, d\tau \, d\tilde{P}(\omega) \\
& + \liminf_{M \to \infty} \lim_{N \to \infty} \int_{\tilde{\Omega}} \int_0^t |B(\tilde{\Psi}^N(\omega, \tau)) - B^M(\tilde{\Psi}^N(\omega, \tau))| \, d\tau \, d\tilde{P}(\omega) \\
& + \liminf_{M \to \infty} \lim_{N \to \infty} \int_{\tilde{\Omega}} \left| \int_0^t B^M(\tilde{\Psi}^N(\omega, \tau)) \, d\tau - \int_0^t B^M(\tilde{\Psi}(\omega, \tau)) \, d\tau \right| \, d\tilde{P}(\omega) \\
& + \liminf_{M \to \infty} \int_{\tilde{\Omega}} \int_0^t |B^M(\tilde{\Psi}(\omega, \tau)) - B(\tilde{\Psi}(\omega, \tau))| \, d\tau \, d\tilde{P}(\omega) \\
& := I_1 + I_2 + I_3 + I_4
\end{align*}
By H\"older's inequality in \eqref{I4Holder}, the invariance \eqref{FullFlowInvariance} of $\rho$ under $\tilde{\Psi}$ in \eqref{I4Invariance}, and Proposition \ref{SumTailGoesToZero} in \eqref{I4Prop42}, we have that
\begin{align}
I_4 & = \liminf_{M \to \infty} \int_{\tilde{\Omega}} \int_0^t |B^M(\tilde{\Psi}(\omega, \tau)) - B(\tilde{\Psi}(\omega, \tau))| \, d\tau \, d\tilde{P}(\omega) \notag \\
& \leq \liminf_{M \to \infty} \int_0^t \left(\int_{\tilde{\Omega}} |B^M(\tilde{\Psi}(\omega, \tau)) - B(\tilde{\Psi}(\omega, \tau))|^2 \, d\tilde{P}(\omega) \right)^\frac12 \, d\tau \label{I4Holder} \\
& \leq \liminf_{M \to \infty} T \left( \int_{X^{-2}} |B^M(\psi) - B(\psi)|^2 \, d\rho(\psi) \right)^\frac12 \label{I4Invariance} \\
& = 0. \label{I4Prop42} 
\end{align}
In order to bound $I_1$, first we fix $N$ and consider the chain of inequalities below. 
\begin{align}
& \quad\; \int_{\tilde{\Omega}} \int_0^t |B^N(\tilde{\Psi}^N(\omega, \tau)) - B(\tilde{\Psi}^N(\omega, \tau))| \, d\tau \, d\tilde{P}(\omega) \\
& = \int_0^t \int_{\tilde \Omega} |B^N(\tilde{\Psi}^N(\omega, \tau)) - B(\tilde{\Psi}^N(\omega, \tau))| \, d\tilde{P}(\omega) \, d\tau \label{I1Fubini} \\
& \leq T^{1/2} \left( \int_0^T \int_{\tilde{\Omega}} |B^N(\tilde{\Psi}^N(\omega, \tau)) - B(\tilde{\Psi}^N(\omega, \tau))|^2 \, d\tilde{P}(\omega) \, d\tau \right)^\frac12 \label{I1CS} \\
& = T^{1/2} \left( \int_0^T \int_{X^{-2}} | B^N(\Psi^N(\psi, \tau)) - B(\Psi^N(\psi, \tau))|^2 \,d\rho(\psi) \, d\tau \right)^\frac12 \label{I1SameLaw} \\
& = T \left( \int_{X^{-2}} |B^N(\psi) - B(\psi)|^2 \, d\rho(\psi) \right)^\frac12, \label{I1Invariance}
\end{align}
where we used the Fubini-Tonelli theorem in \eqref{I1Fubini} followed by Cauchy-Schwartz in \eqref{I1CS}, after which
we used from the Skorokhod Lemma that $\mathcal{L}(\tilde{\Psi}^N) = \mathcal{L}(\Psi^N)$ along with the definition of $\nu_N$ in \eqref{I1SameLaw}, and the fact that $\rho$ is invariant under $\Psi^N$ in \eqref{I1Invariance}. Now we let $N \rightarrow \infty$ and apply Proposition \ref{SumTailGoesToZero} to conclude: 
$$
I_1 = 0. 
$$

Next we estimate $I_2$. For a fixed $M$, we repeat the argument above treating $I_4$ to obtain 
the following identity and estimate: 
\begin{align}
& \quad\; \int_{\tilde{\Omega}} \int_0^t |B(\tilde{\Psi}^N(\omega, \tau)) - B^M(\tilde{\Psi}^N(\omega, \tau))| \, d\tau \, d\tilde{P}(\omega) \label{I2WithN} \\ 
& = \int_0^t \int_{X^{-2}} |B(\psi) - B^M(\psi)| \, d\rho(\psi) \, d\tau \label{I2Identity} \\ 
& \leq T \left( \int_{X^{-2}} |B(\psi) - B^M(\psi)|^2 d\rho(\psi) \right)^\frac12. \label{I2Estimate}
\end{align}
Notice that the equality \eqref{I2Identity} implies that \eqref{I2WithN} is independent of $N$. Therefore to evaluate $I_2$ we need only take the limit inferior $M \to \infty$ in $I_2$. But then applying Proposition \ref{SumTailGoesToZero} to \eqref{I2Estimate} implies that
$$ 
I_2 = 0.
$$

Before we estimate $I_3$, observe that for each fixed $M$, the mapping $\psi \mapsto B_M(\psi)$ is continuous on $X^{-2}$
thanks to the fact that the projection $\Pi_M$ makes the problem completely finite dimensional (the continuity is not uniform with respect to $M$). Recalling that the Skorokhod Lemma gives us that $\tilde{\Psi}^N(\omega) \to \tilde{\Psi}(\omega)$ in $C([0, T] : X^{-2})$, a.s. in $\tilde{\Omega}$, we conclude that for almost every $\omega \in \Omega$,
\begin{equation}
\int_0^t B^M(\tilde{\Psi}^N(\omega, \tau)) \, d\tau \to \int_0^t B^M(\tilde{\Psi}(\omega, \tau)) \, d\tau \qquad \text{as }N \to \infty.
\end{equation}
However, to obtain desired $L^1(\tilde{P})$ convergence, we will apply the Vitali Convergence Theorem (c.f. Theorem A.3.2 of \cite{bog}). To do that 
we show that, for each fixed $M$, the family of functions $$F_N(\omega) := \int_0^t B^M(\tilde{\Psi}^N(\omega, \tau)) \, d\tau$$ is equiintegrable with respect to the probability space $(\tilde{\Omega}, \tilde{\mathcal{F}}, \tilde{P})$.
That is, we must verify the following properties:
\begin{itemize}
\item{The $F_N$ are uniformly bounded in $N$ in $L^1(\tilde{P})$,}
\item{The $F_N$ satisfy $\lim_{\Lambda \to \infty} \sup_{N \in \mathbb{N}} \tilde{P}(\{ \omega \in \tilde{\Omega} : |F_N(\omega)| \geq \Lambda \}) = 0$.}
\end{itemize}
First we have
\begin{align}
\|F_N\|_{L^1(\tilde{P})} & \leq \|F_N\|_{L^2(\tilde{P})} \nonumber \\
& = \left( \int_{\tilde{\Omega}} \left| \int_0^t B^M(\tilde{\Psi}^N(\omega, t)) \, d\tau \right|^2 \, d\tilde{P}(\omega) \right)^\frac12 \nonumber \\
& \leq T \left( \int_{X^{-2}} \left| B^M(\psi)\right|^2 \, d\rho(\psi) \right)^\frac12 \label{equint-M}\\
& < \infty \label{equint-exp}
\end{align}
where in order to obtain \eqref{equint-M} we used Minkowski inequality, and to obtain \eqref{equint-exp} we used the expectation result. 
Similarly by Chebyshev's inequality we have
\begin{align*}
\tilde{P}(\{ \omega \in \tilde{\Omega}: |F_N(\omega)| \geq \Lambda \}) & \leq \frac{\|F_N\|_{L^2(\tilde{P})}^2}{\Lambda^2}
\end{align*}
and since we have already shown that $\|F_N\|_{L^2(\tilde{P})}$ is bounded uniformly in $N$, the equiintegrability of the $F_N$'s follows. Then by the Vitali Convergence Theorem we conclude that
\begin{equation}
\int_0^t B^M(\tilde{\Psi}^N(\omega, \tau)) \, d\tau \to \int_0^t B^M(\tilde{\Psi}(\omega, \tau)) \, d\tau \qquad \text{in } L^1(\tilde{P}) \text{ as }N \to \infty
\end{equation}
But since the above limit holds for each fixed $M$, we have $I_3 = 0$.
\end{proof}

\begin{appendix}

\section{Summary of Calculation for Non-Regularized Streamline Formulation}
In this section we indicate the main steps of the calculation for the expectation for the nonlinear term corresponding to the stream line formulation \eqref{mSQG-str}. We recall the equation \eqref{mSQG-str}

\begin{equation}
\begin{cases}
(|D|^{1 + \delta}\varphi)_t + (u \cdot \nabla)|D|^{1 + \delta}\varphi = 0,\\
u = \nabla^\perp \varphi.
\end{cases}
\end{equation}

Recast the nonlinearity as 

\begin{equation}
B(\varphi, \varphi) = -|D|^{1 + \delta}(\nabla^\perp \varphi \cdot \nabla)|D|^{1 + \delta} \varphi.
\end{equation}

Expand both $\varphi = \sum_k \varphi_k e_k$ and $B = \sum_k B_k e_k$ in the usual $L^2$ orthonormal basis and find

\begin{align*}
B_k & = |k|^{1 + \delta} \sum_{h + h^\prime = k} (h^\perp \cdot h^\prime)|h^\prime|^{1 + \delta} \varphi_h \varphi_{h^\prime} \\
& = |k|^{1 + \delta} \sum_h (h^\perp \cdot k)|k - h|^{1 + \delta} \varphi_h \varphi_{k - h} \\
& = \frac12 |k|^{1 + \delta} \sum_h (h^\perp \cdot k)(|k - h|^{1 + \delta} - |h|^{1 + \delta}) \varphi_h \varphi_{k - h}.
\end{align*}

The conserved enstrophy is now $\|\,|D|^{1 + \delta} \varphi\|_{L^2}$, and so we adjust the correlation for the Gaussian measure accordingly,
\begin{equation}
\rho(M) = \left(\prod_{|k|\leq N} \frac{1}{\sqrt{2\pi |k|^{2 + 2\delta}}} \right) \int_F \exp\left(-\frac12 \sum_{|k|\leq N} |k|^{-2 - 2\delta} |\varphi_k|^2 \right) \, d\varphi_1 \cdots d\varphi_{(2N + 1)^2},
\end{equation}
from which the moment expectations now read

\begin{align}
\mathbb{E}_\rho(\varphi_k) & = 0, \notag \\
\mathbb{E}_\rho(\varphi_k \varphi_{k^\prime}) & = 0, \\
\mathbb{E}_\rho(\varphi_k \overline{\varphi}_{k^\prime}) & = \frac{2\delta_{k, k^\prime}}{|k|^{2 + 2\delta}|k^\prime|^{2 + 2\delta}}.
\notag
\end{align}

Following the proof of Proposition 4.1  we  determine the convergence of the sum

\begin{equation}
\sum_{k} |k|^{2s - 2\delta} \sum_h \frac{\left(\frac{h^\perp}{|h|} \cdot \frac{k}{|k|}\right)^2(|k - h|^{1 + \delta} - |h|^{1 + \delta})^2}{|h|^{2\delta}|h - k|^{2 + 2\delta}}.
\end{equation}

Again using the analogue of Lemma \ref{DeltaDifference}, we have the bound 
$$|k - h|^{1 + \delta} - |h|^{1 + \delta} \leq C|k||h|^\delta,$$
 except this time we do not have to restrict this estimate to high frequencies since the exponent is greater than 1.  For high frequencies, where $|h - k| \sim |h|$, we have that the inner sum is at worst 
 $$\sum_{|h|\geq 2|k|} \frac{|k|^2|h|^{2\delta}}{|h|^{2 + 4\delta}} \lesssim |k|^{2 - 2\delta}.$$ 
 Using the same estimate, the low frequency sum is at worst 
 $$\sum_{|h| \leq 2|k|} \frac{|k|^2 |h|^{2\delta}}{|h|^{2\delta}|k - h|^{2 + 2\delta}} \lesssim |k|^2 \sum_{h \neq 0} \frac{1}{|k - h|^{2 + 2\delta}} \leq  |k|^2.$$  
 Thus together we have
\begin{equation}
\mathbb{E}_\rho (\|B\|_{H^s}^2) \lesssim \sum_k |k|^{2s - 2\delta}(|k|^2 + |k|^{2 - 2\delta} \leq \sum_k |k|^{2s - 2\delta + 2},
\end{equation}
which converges provided $s < -2 + \delta$.

\end{appendix}



\end{document}